\documentclass[12pt, fleqn]{article}
\usepackage[cp1251]{inputenc}
\usepackage{latexsym,amsfonts,amssymb}
\usepackage{graphicx}

\usepackage{amsbsy}
\usepackage{amsmath}
\usepackage{epsf}
\usepackage{cite}
\usepackage{hyperref,color}

\newtheorem{theo}{Theorem}
\newtheorem{remark}{Remark}

\newcommand{\bt}{\begin{theo}}
\newcommand{\et}{\end{theo}}
\newcommand{\bd}{\begin{displaymath}}
\newcommand{\ed}{\end{displaymath}}

\newcommand{\be} {\begin{equation}}
\newcommand{\ee} {\end{equation}}
\newcommand{\ba}{\begin{array}}
\newcommand{\ea} {\end{array}}
\newcommand{\bea}{\begin{eqnarray}}
\newcommand{\eea} {\end{eqnarray}}

\newcommand{\p} {\partial}

\newcommand{\lbd} {\lambda}

\begin{document}

\begin{center}
{\Large \bf
 Multidimensional parabolic equations \\ with a logarithmic nonlinearity \\
 admitting  exceptional Lie symmetries  }

\medskip

{\bf Roman Cherniha $^{a,b,}$\footnote{\small  Corresponding author.
E-mail: r.m.cherniha@gmail.com} and John R. King$^{a}$ }

 $^{a}$ \quad School of Mathematical Sciences, University of Nottingham,\\
  University Park, Nottingham NG7 2RD, UK \\
$^{b}$ \quad National University of Kyiv-Mohyla Academy,\\
2, Skovoroda Street, Kyiv  04070, Ukraine.\\

\end{center}

\begin{abstract}
Nonlinear reaction-diffusion equations arise in a very wide range of applications. Here we address
 a specific multidimensional class in which the source term is non-autonomous in
 time,
 this choice being motivated  both  by its relevance to a broad class of (autonomous) `near-linear'
  reaction-diffusion equations and by the unusually rich nature of the Lie symmetries
  that results (including for arbitrary time-dependence, and with four even more exceptional
  cases arising).  The class of non-autonomous equation analysed   is derived by
  an  asymptotic  approach, with the Lie symmetry classification and a further analysis subsequently
    carried  out.
     Both Lie and non-Lie reductions to ordinary differential equations
are also  noted and relevant exact solutions are constructed.
\end{abstract}


\section{\bf  Introduction} \label{sec:1}

In this work, we study the following class of   evolutionary
multidimensional equations with the logarithmic nonlinearity
\begin{equation}\label{1-1} u_{t}=d(t)\nabla^2 u -f(t)u\ln u , \end{equation}
where $u(t,\bar{x})$ is  unknown function of $n+1$ variables,
$\bar{x}=(x_1,\dots,x_n)$, \  $\nabla=(\frac{\p}{\p x_1},\dots, \frac{\p}{\p x_n})$
% the dot denotes the  scalar product
 and $f(t)$ and $d(t)$ are arbitrary smooth non-zero functions.

First of all, it can be noted that  the diffusivity $d(t)$ is
reducible to $d(t)=1$ without losing  generality by the
transformation
\begin{equation}\label{1-2} t'= \int d(t)dt\equiv D(t), \ f'(t')=\frac{f(D^{-1}(t'))}{d(D^{-1}(t'))}. \end{equation}
So, omitting the primes, in what follows  we consider  the class of
 semilinear reaction-diffusion (RD) equations
\begin{equation}\label{1-0} u_{t}=\nabla^2 u -f(t)u\ln u  \end{equation}
with a single arbitrary function $f(t)\not=0.$
Although the nonlinearity in the RD equation (\ref{1-0}) is specified,  it will be demonstrated below that the logarithmic
 nonlinearity plays a generic role for a specific class of autonomous RD equations.
% with sufficiently general form of sources/sinks.
Here we only note that this nonlinearity naturally arises in population and  cancer modelling and the corresponding
ordinary differential equation
\[U'(t)=-f(t)U\ln U, \]
usually with $f(t)=1$, is known as the Gompertz equation (hereafter
primes denote differentiation w.r.t. $t$). Notably, B. Gompertz
\cite{gompertz-1825} introduced a function related to the above ODE,
while much later
 both that  function and the ODE were applied
  in  cancer modelling \cite{laird-1964}, \cite{kuang-et-al-16}(see Chapter 2 and references
   cited therein) and
  some other biomedical processes \cite{murray-1986},  \cite{murray-03}
   (see Chapter 4 and references cited therein).

Nowadays,   symmetry-based methods  are very common and  powerful
methods for investigation of  nonlinear PDEs, in particular  for the
construction of exact solutions.  Although there are several related
methods (see the well-known book \cite{bl-anco-10} for details), the
classical Lie method
  still remains   the most popular.
There are thousands of published papers devoted to the application
of this method  to PDEs; therefore, we list only some
recent monographs, namely \cite{bl-anco-10, arrigo-15,
ch-se-pl-2018}.  One of the most challenging problems of the Lie method
% arise when there is the
is  performing the   {\it  Lie symmetry classification} (another
terminology is the Lie group classification) of a given nonlinear
PDE. Such problems  occur when the  given  PDE  involves quantities/parameters
(with physical, biological, chemical etc. meanings) that  are
arbitrary functions. It should be stressed that  typically the  Lie
symmetry classification (LSC) is a highly non-trivial problem (see
Chapter 2 in \cite{ch-se-pl-2018} and references cited therein).
Especially, this problem is difficult for    multidimensional
PDEs  because the  results of the LSC can depend on the dimensionality of the space
of independent variables. Typical examples are nonlinear
reaction-diffusion equations \cite{dor-svir-1983} and those
involving convection terms \cite{ch-se-98,ch-se-pr-21}.

The remaining sections are  organized as follows. In
Section~\ref{sec:2}, it is shown how a general RD equation with a
source/sink  involving a specific class of  arbitrary functions can
be  reduced to  (\ref{1-0}) using a reasonable assumption  and a
standard technique of
% approximation
asymptotic  theory.  In Section~\ref{sec:3},
 the LSC problem   is solved for the class
of  RD equations  (\ref{1-0}). It is shown that there are exactly
four forms of the function $f(t)\not=0$  leading to highly
non-trivial Lie symmetries. While the case $f(t)=\mathtt{constant}$
is well-known \cite{dor-svir-1983}, the other three cases are,  to
the best of our knowledge, identified for the first time. In
Section~\ref{sec:4}, some analysis of the Lie algebras obtained is
presented in the case $n=1$ and inequivalent reductions, including
non-Lie ones, of the RD equations of the form (\ref{1-0}) to
ordinary differential equations are constructed. Exact solutions are
derived as well. Finally, we discuss the results obtained and
present some conclusions  in the last section.

\section{\bf Why is the Gompertz  nonlinearity of more general relevance? } \label{sec:2}

Semilinear reaction-diffusion (RD) equations of the form
\begin{equation}\label{0-1} \phi_{t}=\nabla^2 \phi +S(\phi) \end{equation}
have, of course, an enormous number of applications.
Here $\phi(t,\bar{x})$ is  unknown function of $n+1$ variables and  $S$ is an arbitrary function.
% $\bar{x}=(x_1,\dots,x_n)$, \  $\nabla=(\frac{\p}{\p x_1},\dots, \frac{\p}{\p x_n})$.
As noted above, the Gompertz  nonlinearity
\[  S(\phi)=-\phi\ln\phi \]
is of specific interest in biological applications, and is also of mathematical significance, with regard both
to the associated Lie symmetry and, more  importantly in the current context, with its near-borderline status in terms of  the Osgood criterion, which was introduced in the seminal work \cite{osgood} . The latter property provides a key motivation for the current study. Thus, we start from the choice
\begin{equation}\label{0-2}  S(\phi)=\frac{\phi}{\delta}F(-\delta\ln\phi) \end{equation}
in which the otherwise arbitrary function $F(\Psi)$ depends sub-exponentially on $\Psi$ and where $\delta$ (with $0<\delta\ll 1$) is introduced to analyse the evolution from the small amplitude initial data
\[ t=0: \ \phi = e^{-1/\delta}u_0(\bar{x}), \]
where the specified
% given ???????????
 function $u_0(x)$ does not depend on $\delta$.
 The asymptotic ansatz appropriate to the initial data then takes the form
 \[ \phi \sim  e^{-\sigma(t)/\delta}u(t,x) \]
 for $t=O(1)$.
 So,  formula (\ref{0-2}) can be rewritten as
 \[ S(\phi) \sim \frac{\phi}{\delta} F(\sigma(t)-\delta \ln u)=\frac{\phi}{\delta}\Big(F(\sigma)-\delta F'(\sigma)\ln u+o(\delta)\Big). \]
 As a result, after simple calculations,  the RD equation (\ref{0-1}) is reducible to the system
  \begin{equation}\label{0-4} \ba{l}
  \frac{d\sigma}{dt}= - F(\sigma),\\
  \frac{\p u}{\p t}=\nabla^2 u - F'(\sigma)u\ln u, \ea \end{equation}
  subject to  the initial data
  \[ t=0: \ \sigma=1, \ u= u_0(\bar{x}). \]
  Now one realizes that the Gompertz nonlinearity attains
  %obtains
   more generic status\footnote{\small  It is important to stress that account must be taken
    of the full nonlinearity for subsequent stages of the evolution,
    but we shall not concern ourselves with these here.
    See \cite{fa-ki-26} for related considerations}.
  This asymptotic limit therefore motivates our study below of the non-autonomous nonlinear RD equation (\ref{1-0}) with $f(t)=F'(\sigma)$.

 To illustrate the nature of this set up, it suffices to consider two special cases:\\
 \textbf{(i)} $F(\sigma)=\sigma^k$.
 In this case, the first of (\ref{0-4}) implies  $\sigma = \Big(1+(k-1)t\Big)^{\frac{1}{1-k}}$ and hence
 \[ f(t)=
 \frac{k}{1+(k-1)t}. \]
  \textbf{(ii)} $F(\sigma)=\sigma (1+\ln \sigma)^k$.
  In this case   $\sigma = \exp\Big(\big(1+(k-1)t\big)^{\frac{1}{1-k}}-1\Big)$  and hence
  \[f(t)= \Big(1+(k-1)t\Big)^{\frac{k}{1-k}}+\frac{k}{1+(k-1)t}. \]
  In each case $k=1$ plays a borderline role, as is to be expected from the Osgood criterion, and
  accessing arbitrary $f(t)$ requires allowing arbitrarily close approaches to Osgood criticality,
   i.e. broad $f(t)$ correspond to narrow $F(\sigma)$.

\section{\bf  The main result} \label{sec:3}

Here our aim is  solving {\it the Lie symmetry classification
problem}, i.e. to identify  all possible forms of the function
$f(t)$ other than $f(t)\equiv 0$  leading to extensions of the
principal algebra. The latter is the Lie algebra that  consists of
all possible Lie symmetries of  equations of the form (\ref{1-0})
independently of the function $f(t)$.

\begin{theo} \label{th1}
Equation (\ref{1-0}) with an arbitrary function $f(t)$ is invariant
under the  principal  algebra, which is the
$(1+\frac{1}{2}n(n+3))$-dimensional Lie algebra generated by the Lie
symmetry operators
\begin{equation}\label{2-1}\ba{l}
\medskip
P_a= \partial_{x_a}\equiv \frac{\partial}{\partial{x_a}}, \
J_{ab}=x_a\p_{x_b}-x_b\p_{x_a}, \
 a<b=1,\dots,n, \
I= \frac{1}{2}
%\exp\big(-\int f(t)dt\big)
\emph{e}^{-\int f(t)dt}u\p_u, \\
 {\cal{G}}_a=\Big(\int \emph{e}^{-\int
f(t)dt}dt\Big) \partial_{x_a} - \frac{x_a}{2}\emph{e}^{-\int
f(t)dt}u\p_u, \ a=1,\dots,n. \ea
\end{equation}
\end{theo}

%It turns out that there are
\begin{theo} \label{th2}
 Equation (\ref{1-0}) depending on the form of  $f(t)$
admits only  four extensions of the principal algebra (\ref{2-1}).
All these algebras are  $(2+\frac{1}{2}n(n+3))$-dimensional Lie
algebras but with different structures. The relevant basic operators
are listed below.\\
Case 1.$f(t)=f_0=\mathtt{const}$:
\begin{equation}\label{2-2}\ba{l}
\medskip
P_a= \partial_{x_a}\equiv \frac{\partial}{\partial{x_a}}, \
J_{ab}=x_a\p_{x_b}-x_b\p_{x_a}, \
 a<b=1,\dots,n, \
I= \frac{1}{2}e^{-f_0t}u\p_u, \\
%\medskip\\
{\cal{G}}_a= e^{-f_0t}\partial_{x_a}+ \frac{x_a}{2}e^{-f_0t}u\p_u, \
a=1,\dots,n, \  P_t= \partial_{t}. \ea \end{equation}
 Case 2.$f(t)=t^{-1}$:
\begin{equation}\label{2-3}\ba{l}
\medskip
P_a= \partial_{x_a}\equiv \frac{\partial}{\partial{x_a}}, \
J_{ab}=x_a\p_{x_b}-x_b\p_{x_a}, \
 a<b=1,\dots,n, \
I=\frac{1}{2} t^{-1}u\p_u, \\
 {\cal{G}}_a= \ln t\partial_{x_a}- \frac{x_a}{2}t^{-1}u\p_u, \
a=1,\dots,n, \  D=
2t\partial_{t}+x_1\partial_{x_1}+\dots+x_n\partial_{x_n}. \ea
\end{equation}
Case 3.$f(t)=At^{-1}, \ A\not=1$:
\begin{equation}\label{2-4}\ba{l}
\medskip
P_a= \partial_{x_a}\equiv \frac{\partial}{\partial{x_a}}, \
J_{ab}=x_a\p_{x_b}-x_b\p_{x_a}, \
 a<b=1,\dots,n, \
I=\frac{1}{2} t^{-A}u\p_u, \\
{\cal{G}}_a=t^{1-A} \partial_{x_a}- (1-A)\frac{x_a}{2}t^{-A}u\p_u, \
a=1,\dots,n, \  D=
2t\partial_{t}+x_1\partial_{x_1}+\dots+x_n\partial_{x_n}. \ea
\end{equation}
Case 4. If $f(t)$ is a solution of the third-order ODE
\begin{equation}\label{2-5a}\frac{d^3}{dt^3}\big(f^{-1}\big)+f\frac{d^2}{dt^2}\big(f^{-1}\big)=0,
\  \frac{d^2f}{dt^2}\not=0
\end{equation}
%with the restriction
then the operator
\begin{equation}\label{2-5} \Pi= 2g(t)\partial_{t}+\frac{dg(t)}{dt}\big(
x_1\partial_{x_1}+\dots+x_n\partial_{x_n}\big)-\Big(\frac{d^2g(t)}{dt^2}\frac{|\bar
x|^2}{4}+\frac{n}{2}\emph{e}^{-\int
f(t)dt}\int \emph{e}^{\int
f(t)dt}\ \frac{d^2g(t)}{dt^2}dt\Big)u\p_u
\end{equation} and those listed in (\ref{2-1}) apply. In (\ref{2-5}), the
notations  $g(t)=\frac{1}{f(t)}\equiv f^{-1}$ and  $ |\bar
x|^2=x_1^2+\dots+x_n^2$ are used.
\end{theo}

\textbf{Proof of Theorems~\ref{th1} and \ref{th2}.} Here we use the
algorithm, which was suggested in  \cite[Chapter 2]{ch-se-pl-2018}
for the LSC of evolutionary equations. So, we consider  (\ref{1-0}) as a class of RD equations, in which the  Lie symmetry
of each equation with a specified function $f(t)$ may depend on its form.
%Because the group of ETs is already identified,

First of all, we need to find the principal algebra of the given class, i.e. the Lie algebra that is allowed by any equation from this class, independently of $f(t)$.
So, we start from to the most general form of Lie symmetry of the given class:
 \be \label{2-7}
X=\xi^0(t,\bar x,u)\partial_t+\xi^{a}(t,\bar x,u)\partial_{x_a} +\eta(t,\bar x,u)\partial_{u},\ee
where  $\xi^i, \ i=0,\dots,n$ and  $\eta$ are
to-be-determined  functions and a summation over the repeated indexes $a=1,\dots,n$ is assumed. The well-known infinitesimal criterion
of invariance of (\ref{1-0}) with respect to the symmetry $X$ reads
as
\[ \mbox{\raisebox{-1.6ex}{$\stackrel{\displaystyle
X}{\scriptstyle 2}$}}\,
\Big( u_{t}-\nabla^2 u +f(t)u\ln u  \Big)\Big\vert_{\cal{M}}=0,
\]
where the operator $
\mbox{\raisebox{-1.6ex}{$\stackrel{\displaystyle X}{\scriptstyle
2}$}}$ is the second-order prolongation of the operator $X$, and
the manifold   ${\cal{M}}$ consists of the equation in question, i.e.
\[ {\cal{M}}= \{ u_{t}-\nabla^2 u +f(t)u\ln u=0 \}. \]
Because the second-order prolongation of the operator $X$ is
calculated via the well-known formulae \cite{bl-anco-10, arrigo-15,
ch-se-pl-2018} (actually, these were derived by S.~Lie in his
classical papers), using  the above criterion and making  relevant
computations,  one  obtains to the following  restriction  on
$\xi^i$ and  $\eta$: \be \label{2-8} \xi^0(t,\bar x,u)=\xi^0(t),
\quad \xi^{a}(t,\bar x,u)= \xi^{a}(t,\bar x), \ a=1,\dots,n, \quad
\eta(t,\bar x,u)=r(t,\bar x)u, \ee where the functions on the
right-hand sides remain to be determined. Having formulae
(\ref{2-8}), in which the dependence on $u$ is already identified,
it is easy
%easily !!!
 to derive
the so-called system of determining equations (DEs):
\be \label{2-9} \ba{l}\medskip
\xi^0_t=2\xi^a_{x_a}, \quad  \xi^a_{x_b}+\xi^b_{x_a}=0, \ a\not=b=1,\dots,n \\
\medskip
\xi^a_{t}=\nabla^2\xi^a - 2r_{x_a}, \  a=1,\dots,n \\
r_t= \nabla^2r - f(t)r,  \quad \xi^0(t)f(t)=C,  \ C \in \mathbb{R}.
\ea\ee

In the case $n=1$, the system of DEs (\ref{2-9}) consists of three linear PDEs and the algebraic condition  $\xi^0(t)f(t)=C$,
therefore its integration is much simpler than for $n>1$. In particular, one easily derives three Lie symmetries arising in (\ref{2-1}) in the 1D case.

Let us consider the first non-trivial case occurring for $n=2$, when
(\ref{2-9}) consists of six PDEs. Solving the first three equations,
we obtain \be \label{2-10}   \xi^0=2g(t), \
\xi^1=\frac{dg}{dt}x_1+A(t)x_2+B_1(t), \
\xi^2=\frac{dg}{dt}x_2-A(t)x_1+B_2(t), \ee where the functions $g, \
A, B_1$ and $B_2$ need to be determined from the remaining three
PDEs involving the functions $r$ and $f$. Notably,
$g(t)f(t)=\frac{C}{2}$  follows from the last line of (\ref{2-9}).
Assuming that $f(t)$ is an arbitrary smooth function, we immediately
derive $g(t)=0$ from the algebraic condition and then the three
equations listed in the second and third lines of (\ref{2-9}) are
reducible to the linear ODEs, leading to
\[ \ba{l}\medskip  B_1(t)=C_1\emph{e}^{-\int f(t)dt}+C_3, \  B_2(t)=C_2 \emph{e}^{-\int f(t)dt}+C_4,  \ A(t)=  C_5,\\
r(t,\bar{x})=-\frac12 \emph{e}^{-\int
f(t)dt}(C_1x_1+C_2x_2)+C_6\emph{e}^{-\int f(t)dt}, \ea \] where
$C_i, i=1,\dots,6$ are arbitrary constants. Using the above formulae
and (\ref{2-9}), we obtain the most general form of Lie symmetries
provided $f(t)$ is  arbitrary. As a result, the six-dimensional Lie
algebra arising in Theorems~\ref{th1} in the case $n=2$ is obtained.

Now we need to identify all possible extensions of this Lie algebra depending on $f(t)$.
In this case, $g(t)\not=0$ (see (\ref{2-10})) and then $g(t)=\frac{C}{2f(t)}$, while  $A(t)$ is again a constant, and the system of  ODEs
\be \label{2-11}  \frac{d^3g}{dt^3} +f\frac{d^2g}{dt^2}=0, \ \frac{d^2B_a}{dt^2} +f\frac{dB_a}{dt}=0, \ a=1,2, \ \frac{dD}{dt}+fD=-\frac{d^2g}{dt^2} \ee
must be solved. Having done this, one immediately obtains
\be \label{2-12}  r(t,\bar{x})= -\frac14 \frac{d^2g}{dt^2}|\bar{x}|^2 - \frac{dB_a}{dt}x_a +D(t). \ee
Assuming $\frac{d^2g}{dt^2}=0$, the ODE system (\ref{2-11}) can easily be  integrated and, taking into account
the algebraic condition  $\xi^0(t)f(t)=C$, we obtain exactly the operators arising in Cases 1, 2 and 3 of
Theorems~\ref{th2}. To be precise,  the functions $f(t)=(t+t_0)^{-1}$ and $f(t)= A(t+t_0)^{-1}$ are obtained in Cases
 2 and 3, however, the   class of RD equations (\ref{1-0}) admits the equivalence transformation $t+t_0 \longrightarrow t$.

The case $\frac{d^2g}{dt^2}\not=0$ is  highly non-trivial because
the first ODE in (\ref{2-11}) is equivalent to \be \label{2-13}
g\frac{d^3g}{dt^3} +\frac{C}{ 2}\frac{d^2g}{dt^2}=0, \ee which is
integrable in a parametric form  only (see Appendix A). Notably, we
may set $C=2$ without losing generality because $t \longrightarrow
|C|t, \  \bar{x} \longrightarrow \sqrt{|C|}\bar{x},  \  f
\longrightarrow |C|f$ is an equivalence transformation. According to
the general theory of ODEs, the above third-order ODE must possess
solutions involving three arbitrary constants,  not just  those of
linear form. Thus, taking  $g(t)=\frac{1}{f(t)}$ to  satisfy this
ODE and  taking into account (\ref{2-10})  and (\ref{2-12}),     we
obtain an additional Lie symmetry of (\ref{1-0})
\begin{equation}\label{2-5*} \Pi= 2g(t)\partial_{t}+\frac{dg(t)}{dt}\big(
x_1\partial_{x_1}+x_2\partial_{x_2}\big)-\Big(\frac{d^2g(t)}{dt^2}\frac{|\bar
x|^2}{4}+\emph{e}^{-\int
f(t)dt}\int \emph{e}^{\int
f(t)dt}\ \frac{d^2g(t)}{dt^2}dt\Big)u\p_u.
\end{equation}
This means Case 4 for $n=2$ is identified. For $n>2$, each step of the proof is a quite similar.
In particular, $n-1$ terms are obtained in (\ref{2-10}) instead of $A(t)x_a$
 and the term $-\frac{n}{2} \frac{d^2g}{dt^2}$ springs up in the right-hand-side
 of the last of (\ref{2-11}) instead of $-\frac{d^2g}{dt^2}$.

%\blacksquare
$\blacksquare$

\begin{remark}
Each function $f(t)$ arising in Cases 2--4 can be slightly
generalised by the equivalence transformation $ t \to t+ t_0, \  t_0
\in \mathbb{R}$.
\end{remark}

%\begin{remark}
To the best of our knowledge, the general solution of  ODE
(\ref{2-5a}) was unknown. In particular, this ODE is not listed in the well-known handbooks \cite{kamke,pol-za-2018}.
We were, however,  able to construct its general solution in parametric  form (see Appendix A).
 Because the general solution involves three arbitrary constants,
 %(according to the classical  theory of ODEs),
 it is clear  that  Case 4 is not empty.

It can be verified that  Cases 1--3 can be derived in a formal way
from  Case 4 under the restriction  $\frac{d^2g}{dt^2}=0$. However,
the Lie algebras arising in  Cases 1--4 are essentially different,
therefore all cases should be distinguished from the point of view
of theory of Lie algebras. Moreover, applications of relevant Lie
symmetries for the  search for exact solutions may lead to essentially
different structures of the solutions in Cases 1--4.

\section{\bf  Analysis of Lie algebras and reductions to ordinary differential equations} \label{sec:4}

\subsection{\bf Analysis of Lie algebras}

All the Lie algebras arising in Theorem \ref{th2}  have the same dimensionality $N=(2+\frac{1}{2}n(n+3))$.
However, they are  essentially different algebras. In order to demonstrate this, we analyse commutators (Lie brackets) of their basic operators.  It is enough  to consider 1D space, for which  $N=4$, i.e. four-dimensional algebras are obtained.

\begin{table}[h!]
\caption{Table of commutators of the Lie algebra  (\ref{2-2}) with $n=1$}\medskip
\label{tab1}   \begin{center}
\begin{tabular}{|c|c|c|c|c|}
\hline  & $P_t$  &  $P_x$ & $I$  & ${\cal{G}}$ \\
 \hline &&&& \\
  $P_t$  & $0$ & $0$& $-f_0I$& $-f_0{\cal{G}}$  \\ \hline &&&&\\
  $P_x$  &$0$ & $0$ & $0$& $I$ \\ \hline &&&& \\
   $I$  & $f_0I$& $0$&$0$&$0$  \\ \hline &&&&\\
    ${\cal{G}}$  & $f_0{\cal{G}}$&$-I$ &$0$&  $0$
   \\  \hline
 \end{tabular}\end{center}
 \end{table}

  \begin{table}[h!]
\caption{Table of commutators of the Lie algebra  (\ref{2-3}) with $n=1$}\medskip
\label{tab2}   \begin{center}
\begin{tabular}{|c|c|c|c|c|}
\hline  & $D$  &  $P_x$ & $I$  & ${\cal{G}}$ \\
 \hline &&&& \\
  $D$  & $0$ & $-P_x$& $-2I$& $2P_x-{\cal{G}}$  \\ \hline &&&&\\
  $P_x$  &$P_x$ & $0$ & $0$& $-I$ \\ \hline &&&& \\
   $I$  & $2I$& $0$&$0$&$0$  \\ \hline &&&&\\
    ${\cal{G}}$  & ${\cal{G}}-2P_x$& $I$ &$0$&  $0$
   \\  \hline
 \end{tabular}\end{center}
 \end{table}

 \begin{table}[h!]
\caption{Table of commutators of the Lie algebra  (\ref{2-4}) with $n=1$}\medskip
\label{tab3}   \begin{center}
\begin{tabular}{|c|c|c|c|c|}
\hline  & $D$  &  $P_x$ & $I$  & ${\cal{G}}$ \\
 \hline &&&& \\
  $D$  & $0$ & $-P_x$& $-2AI$& $(1-2A){\cal{G}}$  \\ \hline &&&&\\
  $P_x$  &$P_x$ & $0$ & $0$& $(A-1)I$ \\ \hline &&&& \\
   $I$  & $2AI$& $0$&$0$&$0$  \\ \hline &&&&\\
    ${\cal{G}}$  & $(2A-1){\cal{G}}$& $(1-A)I$ &$0$&  $0$
   \\  \hline
 \end{tabular}\end{center}
 \end{table}

  \begin{table}[h!]
\caption{Table of commutators of the Lie algebra  (\ref{2-1}) and (\ref{2-5}) with $n=1$}\medskip
\label{tab4}   \begin{center}
\begin{tabular}{|c|c|c|c|c|}
\hline  & $\Pi$  &  $P_x$ & $I$  & ${\cal{G}}$ \\
 \hline &&&& \\
  $\Pi$  & $0$ & $-{\cal{G}}$& $-\gamma I$& $-\gamma{\cal{G}}-\beta P_x$  \\ \hline &&&&\\
  $P_x$  &${\cal{G}}$ & $0$ & $0$& $-I$ \\ \hline &&&& \\
   $I$  & $\gamma I$& $0$&$0$&$0$  \\ \hline &&&&\\
    ${\cal{G}}$  & $\gamma{\cal{G}}+\beta P_x$& $I$ &$0$&  $0$
   \\  \hline
 \end{tabular}\end{center}
 \end{table}

\newpage

\begin{remark}
The constants $ \beta$ and $\gamma$ in Table \ref{tab4}   are determined  up to a single arbitrary constant, this corresponding to the nonlinear  solution of the third-order ODE (\ref{2-5a}): the general solution of this equation involves  three arbitrary  constants $C_1, C_2$ and $C_3$ (see Appendix A), but  $C_2$  and $C_3$ are  reducible to $C_2=1$ and $C_3=0$ by scale transformations and time-translations, respectively.
\end{remark}

%It can be easily seen
A simple analysis  of non-zero  Lie brackets in Tables \ref{tab1}--\ref{tab4} leads to a preliminary conclusion that the relevant four-dimensional Lie algebras are inequivalent. We remind the reader that sometimes   Lie algebras of the same dimensionality can be different representations of the same Lie algebra (see, e.g., a highly non-trivial example in \cite{ch-ki-25}). So, we used a seminal work \cite{Pat-Wint}, in which all inequivalent  four-dimensional Lie algebras are analysed,  to explore this issue.
After a careful analysis, we concluded that all the Lie algebras derived here are indeed inequivalent.

The first  Lie algebra presented in  Table \ref{tab1} is nothing else but  a representation of the four-dimensional Lie algebra $A^0_{4,9}$ listed in Table II of \cite{Pat-Wint}. In fact, setting $e_1=I, \ e_2={\cal{G}}, \ e_3=-P_x$
and $e_4=P_t$,  one obtains exactly the same non-zero  Lie brackets  ($f_0$ can be reduced to 1 without losing generality).
The Lie algebra presented  in Table \ref{tab2} is  a representation of the four-dimensional Lie algebra
$A_{4,7}$ listed in Table II of \cite{Pat-Wint}. This  is shown by setting $e_1=I, \ e_2=-2P_x, \ e_3={\cal{G}}$
and $e_4=D$.
The Lie algebra presented  in Table \ref{tab3} is  a representation of the algebra
$A^b_{4,9}$ listed in Table II of \cite{Pat-Wint}. This is shown by setting $e_1=(A-1)I, \ e_2=P_x, \ e_3={\cal{G}}$ and $e_4=D$ (the parameter $b=2A-1$).

Finally, we need to identify the type of the Lie algebra in Table \ref{tab4}. A problem occurs
in this case because we cannot specify the constants  $ \beta$ and $\gamma$, these constants
 in Table \ref{tab4}   being determined  up to a single arbitrary
 constant -- see Remark 2.
% this corresponding to a    solution of the third-order nonlinear ODE (\ref{2-5a}).
 We only note that this algebra with $\beta=-1$ and $\gamma=2$
is equivalent to the Lie algebra $A_{4,7}$ listed in Table II of \cite{Pat-Wint}.
%However, taking into account that  these constants can be expressed in terms of  a single parameter, we assume $that this is a representation of the  algebra
%$A^a_{4,11}$ listed in Table II of  \cite{Pat-Wint}. In particular, setting  $\alpha=1, \
%one easily obtains the representation of $A^a_{4,11}$ with $a=1$.
Further analysis  is needed in order to show rigorously that any pair of  $ \beta$ and $\gamma$
leads to a  Lie algebra listed in \cite{Pat-Wint}. However,  this  is a purely algebraic problem
 and  lies beyond scopes of this study; the form-preserving
 transformations of Section 4.3 seem to shed light on this matter.
% WE CAN SET   gamma=1 ????????????????????????????????????

\subsection{\bf  Lie reductions to ordinary differential equations}

Here we are looking for Lie reductions and exact solutions of the RD equation  (\ref{1-1})
in one space dimensionality, i.e.
\begin{equation}\label{4-1} u_{t}= u_{xx} -f(t)u\ln u. \end{equation}

According to the direct approach for finding exact solutions, one should start from the
 most general linear combination of Lie symmetries, which  reads as
\begin{equation}\label{4-2}
Y=\alpha_1\p_x+ 2\alpha_2 {\cal{G}}+2\alpha_3I,
\end{equation}
where $\alpha_i, \ i=1,2,3$ are arbitrary parameters,
$|\alpha_1|+|\alpha_2|\not=0$ (otherwise a Lie reduction does not
exist)  and
\begin{equation}\label{4-3} I= \frac{1}{2}F(t)u\p_u, \quad
 {\cal{G}}=\Big(\int F(t)dt\Big) \partial_{x} -
\frac{x}{2}F(t)u\p_u, \quad F(t)=\emph{e}^{-\int
f(t)dt}. \end{equation}

It can be easily noted that there are only two  cases leading to
inequivalent reductions: $\alpha_2\not=0$  and $\alpha_2=0$.
Assuming the first restriction, the corresponding invariant surface
condition is
\[ Y(u)=0 \Leftrightarrow (\alpha_1+2\alpha_2\int F(t)dt)u_x =(\alpha_3- \alpha_2 x)u, \]
which can be simplified by invariance transformation $\alpha_2x
-\alpha_3 \rightarrow \alpha_2x $ to the form \textcolor{blue} {
\begin{equation}\label{4-4}
2\Big(\alpha+\int F(t)dt\Big)u_x = -F(t)xu, \quad
\alpha=\frac{\alpha_1}{2\alpha_2}. \ee } So, integrating
(\ref{4-4}), we obtain the ansatz
\begin{equation}\label{4-5}
u= u_0(t)\exp\Big(-\frac{x^2F(t)}{4(\alpha+\int F(t)dt)}\Big),\end{equation}
where the function  $u_0(t)$ to-be-determined.
Thus, substituting ansatz (\ref{4-5}) into (\ref{4-1}), one obtains
\begin{equation}\label{4-5a} \frac{du_0}{dt} + \frac{F(t)}{2(\alpha+\int F(t)dt)}u_0 +f(t)u_0\ln u_0=0. \ee
The above first-order ODE is integrable and its general solution is
\[ u_0(t)=\exp\Big(\beta F(t) -\frac{1}{2}F(t)\int\frac{dt}{\alpha+\int F(t)dt}\Big). \]
Finally, using ansatz  (\ref{4-5}), we arrive at the family of exact solutions of  (\ref{4-1})
\begin{equation}\label{4-6}
u(t,x)=\exp\Big(\beta F(t) -\frac{1}{2}F(t)\int\frac{dt}{\alpha+\int F(t)dt}-\frac{x^2F(t)}{4(\alpha+\int F(t)dt)}\Big), \quad F(t)=\emph{e}^{-\int
f(t)dt},\end{equation}
where $\alpha$ and $\beta$ are arbitrary constants. Because the RD equation  (\ref{1-1}) is invariant under
 space translations, one may introduce into (\ref{4-6}) the third arbitrary parameter: $x \longrightarrow x+x_0$.
 The above family of exact solutions with $\alpha>0$ is of particular
 significance, since it captures the intermediate asymptotics of the
 evolution to the full nonlinearity (cf. footnote 2).

Now we turn to the case $\alpha_2=0$ and then automatically  $\alpha_1\not=0$.
In this case, the relevant ansatz is
\[ u= u_0(t)\exp(\alpha xF(t)),  \quad  \alpha=\frac{\alpha_3}{\alpha_1},\]
which reduces the RD equation  (\ref{1-1}) to an ODE with the same structure as (\ref{4-5a}).
As a result, the following family of exact solutions of  (\ref{4-1}) is derived
\begin{equation}\label{4-7}
u(t,x)=\exp\Big(\beta F(t)+\alpha^2F(t)\int F(t)dt + \alpha xF(t)\Big). \end{equation}

 %\cite{bl-anco-10, arrigo15, ch-se-pl-2018}.

\subsection{\bf  Non-Lie reductions and form-preserving transformations}

It can be noted that the nonlinear equation (\ref{1-0}) admits  non-Lie reductions.
This is worth highlighting since they describe the intermediate asymptotics  of the Cauchy
 problem for the transition from (\ref{1-0})
to (\ref{0-1}) with (\ref{0-2}). In the 1D case the relevant ansatz has the form
\begin{equation}\label{4-10} u= \mathtt{exp}\big(a(t)+b(t)x+c(t)x^2\big), \ee
where $a, \ b$ and $c$ are to-be-determined functions with $c<0$
occuring in typical physical applications). Substituting the above
ansatz into (\ref{4-1}) and making simple calculations, one obtains
the following non-Lie reduction of the PDE in question:
 \begin{equation}\label{4-11}\ba{l}
%\medskip
\frac{da}{dt}=b^2 +2c - fa, \\
\frac{db}{dt}=4bc - fb, \\
\frac{dc}{dt}=4c^2 - fc.
\ea  \end{equation}

It should be noted that the above reduction was derived earlier only in the case $f(t)=\mathtt{constant}$
(see, for example, \cite{gala-svir-book}).

Interestingly, the nonlinear ODE system (\ref{4-11}) is integrable independently of the form
of the function $f(t)$. In fact, it can be noted that the second and the third equations from (\ref{4-11}) lead to the relation $b=2x_0c$ ($x_0$ is an arbitrary constant). So, introducing the new function $A(t)=a(t)- x_0^2c(t)$, the first equation from (\ref{4-11}) takes the linear form
\[ \frac{dA}{dt}=2c - fA. \]
Thus, integrating the third  equation from (\ref{4-11}) (which  is a Bernoulli equation and hence is reducible to a linear equation by the substitution $c(t)=\frac{1}{C(t)}$), we immediately obtain the functions $b$ and $c$. Having  $c$, the function $A$ is obtainable from the above linear ODE.  At the final stage, substituting the functions $a, \ b$ and $c$ into the non-Lie ansatz (\ref{4-10}),
a three-parameter family of exact solutions of the RD equation  (\ref{4-1}) is constructed.
%Surprisingly,
 The family obtained in fact  coincides (up to notations) with the family  of exact solutions   (\ref{4-6}).
It can be highlighted that here we found another confirmation of an old  hypothesis claiming
 that non-Lie reductions of nonlinear PDE admitting a wide Lie symmetry often produce exact solutions that are obtainable via Lie symmetries. This hypothesis, which   was discussed probably for the first time in   \cite{ch-JPhys-98} (see also \cite[Section 4.1]{ch-se-pl-2018}), has no proof; however, there are many examples confirming its broad applicability.

\begin{remark}
The above ansatz can be straightforwardly generalised to the
multidimensional case as follows
\[ u= \mathtt{exp}\big(a(t)+b_i(t)x_i+c_{ij}(t)x_ix_j\big), \]
where summations over the repeated indices  $i$ and $j$ from $1$ to $n$ are assumed.
\end{remark}

 In conclusion of this section, we present a highly non-trivial example of transformations
 for the class  of equations  (\ref{4-1}) allowing the connection  of solutions of two representatives
  of this class,  with the functions  $f$ and $F$ say,  provided  correctly-specified restrictions take place.
 Here we restrict ourselves to $n=1$ and   transform  (\ref{4-1}) using  the transformation
 \begin{equation}\label{4-12}
 u(t,x)= \mathtt{exp}\big(a(t)+b(t)x+c(t)x^2\big)U(t,x).
\end{equation}
%Assuming ???
Provided  that the  functions  $a, \ b$ and $c$  satisfy (\ref{4-11}), one easily notes that   the equation
 \begin{equation}\label{4-13} U_{t}= U_{xx}+(2b(t)+4c(t)x)U_x -f(t)U\ln U\ee
is obtained instead of (\ref{4-1}).
Now we transform the space variable as follows
\[ x=s(t) +\frac{X}{\lambda(t)}, \]
where the functions  $\lambda(t)$ and $s(t)$  form a solution of the
linear ODE system \begin{equation}\label{4-13*}
\frac{d\lambda}{dt}=4c\lambda,  \quad \frac{ds}{dt}=-(2b+4cs). \ee
Having this done, (\ref{4-13})  takes the form
\begin{equation}\label{4-14} U_{t}=\lambda^2(t) U_{XX}-f(t)U\ln U.\ee
At the last stage, we apply the transformation for  the time variable:
\[ T=\int\lambda^2(t)dt\equiv \Lambda(t), \]
which reduces (\ref{4-14})  to the form
\begin{equation}\label{4-15} U_{T}= U_{XX}-F(T) U\ln U, \end{equation}
where $F(T)=\frac {f(\Lambda^{-1}(T))}{\lambda^2(\Lambda^{-1}(T))}$ and $\Lambda^{-1}$ is an inverse function to
$\Lambda$.

 Of course, the above chain of  transformations does
not form in general an equivalence transformation for the class  of
equations (\ref{4-1}). However, one may consider them as  a highly
non-trivial example of form-preserving (admissible) transformations
introduced independently in  \cite{kingston-91} and \cite{ga-wi-92}
for the analysis of evolution equations. Later these transformations
were used for solving LSC problems (see \cite[Section
2.3]{ch-se-pl-2018} and references cited therein).

It should  be pointed out  that the chain of transformations that
maps  (\ref{4-1}) into (\ref{4-15}) is reciprocal
%as can be observed by
in the following sense:  introducing $\Lambda(T)$ and $C(T)$
according to
\[ t= \int\Lambda^2(T)dT, \quad C(T)=-\frac{c(t)}{\lambda^2(t)}, \]
from which follows that $\Lambda(T)=\frac{1}{\lambda(t)}$  and
\[ \frac{dC}{dT}=4C^2 - FC, \quad \frac{d\Lambda}{dT}=4C\Lambda. \]

%All non-conjugated subalgebras of this algebra can
%be found in Table~II of the classical paper \cite{Pat-Wint}.

 Finally, two additional comments are in order. Firstly, for
$f(t)=At^{-1} $ (as in Case 3 of Theorem 2), a particular solution
for $c(t)$ is given by $c(t)= \frac{A-1}{4t},$ so that
\begin{equation}\label{4-16}
\lambda=\lambda_0 t^{A-1}, \quad  T=
\frac{\lambda^2_0}{2A-1}t^{2A-1}, \quad F(T)= \frac{A}{2A-1}T^{-1},
\quad A\not=\frac12. \ee Thus, two equations  with  different values
of $A$ in Case 3 (see Theorem 2) can be mapped into each other.

 There is also the special case $A=\frac12$ leading to
\begin{equation}\label{4-17} T= \lambda^2_0 \ln t, \quad F(T)=
\frac{1}{2\lambda^2_0}\equiv f_0. \ee This means that Case 3 with
$A=\frac12$ is reducible to Case 1 of Theorem 2 and  (\ref{4-17}) is
the corresponding  form-preserving  transformation. Notably, the Lie
algebra $A^b_{4,9}$ (see Subsection 4.1) is automaticaly reducible
to $A^0_{4,9}$ because $b=2A-1$, therefore an exceptional status of
$A=\frac12$ is reflected in $A^b_{4,9}$.

 Secondly, defining $G(T)=\frac{1}{F(T)}$, we have , on applying
the third equation of (\ref{4-11}), that
\[ \frac{dG}{dT}= \frac{dg}{dt}+ 8cg,   \ \frac{d^2G}{dT^2}= \frac{1}{\lambda^2}
\Big(\frac{d^2g}{dt^2} +8c\frac{dg}{dt} +32c^2g -8c\Big), \] from
which it follows that
\begin{equation}\label{4-18}
G\frac{d^2G}{dT^2} -\frac12 \Big(\frac{dG}{dT}\Big)^2 +\frac{dG}{dT}
= g\frac{d^2g}{dt^2} -\frac12 \Big(\frac{dg}{dt}\Big)^2
+\frac{dg}{dt}. \end{equation}
 Thus, a mapping between instances of
Case 4 is derived; moreover, by setting
$g(t)=\frac{1}{f(t)}=\frac{t}{A}$ in (\ref{4-18}), as in Cases 2 and
3 of Theorem 2, we obtain
\[ G\frac{d^2G}{dT^2} -\frac12 \Big(\frac{dG}{dT}\Big)^2 +\frac{dG}{dT}
= \frac{1}{A} - \frac{1}{2A^2}, \] which is equivalent to
(\ref{A-1}) with $B=\frac{1}{A} - \frac{1}{2A^2}$.

In consequence, there is a mapping between Cases 2, 3 and 4 of
Theorem 2; it is significant in this regard that $c(t)$ is any
solution obtained from (\ref{4-11}), so the transformation generates
an additional arbitrary constant. In the terms of the corresponding
Lie algebras, it means that the Lie algebra in Table \ref{tab4} with
correctly-specified values of $\beta$ and $\gamma$  should be
equivalent to those presented in Tables \ref{tab2}  and \ref{tab3}.
Moreover, one may claim that the four sets of equations listed in
Theorem 2 are inequivalent  under  equivalence transformations  of
the RD equation (\ref{1-0}); however, some of them are equivalent
under form-preserving transformations, as shown above.

%All non-conjugated subalgebras of this algebra can
%be found in Table~II of the classical paper \cite{Pat-Wint}.

%\newpage

\section{\bf  Conclusions} \label{sec:5}

The main results of this work can be summarised as follows.
It was identified that the Gompertz-type  equation (\ref{1-0}) involving diffusivity and an arbitrary time-dependent function $f(t)$ admits
 {\it remarkable Lie symmetries that form a $(1+\frac{1}{2}n(n+3))$-dimensional Lie algebra}.
 Moreover, we proved that there are exactly four  specific forms of  $f(t)$ when this algebra becomes to be wider. In other words, a complete LSC was derived. It is noteworthy that the principal algebra of the nonlinear equation (\ref{1-0}) is itself wider than that for other  RD equations involving nonlinear source/sink forms with   arbitrary functions.

We also note that some RD equations involving arbitrary time-dependent functions were studied earlier by means of the classical Lie method.  For example, the class of RD equations
\[u_t = u_{xx} -f(t)u^3 \]
was studied in \cite{van-at-ol-19} and
%only obvious
 all possible Lie  symmetries were found, in
particular, it was proved that the principal algebra admits
extensions only in the cases  $f(t)=f_0t^k$ and
$f(t)=f_0\emph{e}^{kt}, \ k \in \mathbb{R}$.  In the recent paper
\cite{ch-da-25}, an age-structured diffusive model for epidemic
modelling is studied, which is reducible to the form
\begin{equation}\label{5-1} u_{\tau}=\nabla^2 u +f(\tau)u - u^{\kappa} , \end{equation}
by introducing a variable $\tau$ that combines  time  and age. So,
the Lie symmetries  derived in  \cite{ch-da-25}(see Table 1
therein), can easily be transformed  to those for   (\ref{5-1}).
Although, there are  many specific forms of $f(\tau)$ leading to the
relevant extension of the principal algebra of (\ref{5-1}), the
latter, in contrast of (\ref{2-1}), consists only of the generators
$P_a$ and $ J_{ab}$. So, this again underlines a special status of
the Gompertz nonlinearity. Interestingly, the RD equation
(\ref{4-1}) with $f(t)=constant$ is reducible to the Burgers
equation with a time-variable coefficient by a chain of
substitutions \cite{bl-zhe-2013}.

In \cite{bas-la-zh-2001}, an attempt to solve the LSC problem for a very general  class of
quasilinear  evolution equations (in 1D case) was made. However, the authors used a purely
algebraic approach, in which they started from the {\it simplest realisations} of
 the 2-, 3-, 4- and 5-dimensional Lie algebras. As a result, the relevant evolution PDEs obtained
 in  \cite{bas-la-zh-2001} have been found  in cumbersome and awkward forms. It is very difficult
   to identify the class of equations (\ref{1-0}) and even any particular case in the tables
   presented in \cite{bas-la-zh-2001}. There are also several
   rigorous studies in which LSC is performed for classes of
   evolutionary 1D
   equations with space-variable  coefficients (see, e.g.,
   \cite{ga-wi-92},\cite{van-at-al-12}).

  The Gompertz-type  equation (\ref{1-0}) was further studied above in 1D case. All
inequivalent ans\"atze and reductions  were constructed and the ODEs
obtained were completely solved,  exact solutions of the PDE in
question  being found. A   non-Lie  ansatz reducing the equation in
question to a system of ODEs was identified as well. Interestingly,
 the exact solutions constructed using the non-Lie  ansatz coincide
with those derived by a Lie symmetry reduction. Highly non-trivial
examples of  mappings between different members of
 the class  of equations  (\ref{4-1}) are also  constructed and the relevant interpretation
  presented.

Finally, we note a sense in which the nonlinear non-autonomous
equation (\ref{1-0}) is exceptional in inheriting some of the
properties of the linear heat (diffusion)  equation. Limiting our
comments to the 1D space case, the linear heat equation possesses
two symmetries beyond the obvious ones of translation and scaling
invariance and of linear superposition. The first of these
generalises to arbitrary $f(t)$ on setting $c(t)=0$ (so that
$\lbd=1$ without loss of generality, see (\ref{4-13*})) in the form
-preserving transformation of Section 4.3, this providing a direct
derivation of ${\cal{G}}_1$ in (\ref{2-1}). Notably, ${\cal{G}}_1$
with $f(t)=0$ gives exactly the operator of Galilei transformations
for the  heat equation. The second, corresponding to projective
transformations (the Appel transformation), generalises only to the
special cases above, but is also reflected in the
form-transformations with $c(t)\not=0$. It can be also noted that
$\Pi$ (\ref{2-5}) with $f(t)=0$ and $g(t)= t^2$ produces the
operator of projective transformations for the linear heat equation.

\section{Acknowledgement}
%\acknowledgements
R.Ch. acknowledges that this research was partly funded by the Isaac Newton
Institute within the Rebuild Ukraine: Satellite scheme and the University of Nottingham.
 J.R.K.
gratefully acknowledges a Royal Society Leverhulme Trust Senior
Fellowship.  The authors are grateful to anonymous reviewers for
some helpful comments.

\section*{\bf  Appendix A}

The nonlinear ODE
(\ref{2-5a}) with $g(t)=\frac{1}{f(t)}\equiv f^{-1}$ takes the form
\begin{equation}\label{A-1}
gg'''+g''=0  \Leftrightarrow gg'' -\frac12 g'^2+g'=B,
\end{equation}
 where $B$ is an arbitrary constant that is  taken in the form $B=\frac12 (1-C_1^2)$
  in what follows.
Introducing the function $p(g)=g'$, one obtains the first-order ODE with separable variables
\[ gp\frac{dp}{dg}= \frac12(p- (1-C_1))(p- (1+C_1))\]
that can easily be integrated. As a result, we get
 \[ (1+C_1)\ln(p-(1+C_1))- (1-C_1)\ln(p-(1-C_1))=C_1(\ln g - \ln C_2) \]
 i.e.
\begin{equation}\label{A-2}
g = C_2 (p-(1+C_1))^{1+\frac{1}{C_1}}(p-(1-C_1))^{1-\frac{1}{C_1}},
\end{equation}
where $C_2$ is another arbitrary non-zero constant.
On the other hand, $gp\frac{dp}{dg}= gp'$, therefore
 \[ C_2 (p-(1+C_1))^{1+\frac{1}{C_1}}(p-(1-C_1))^{1-\frac{1}{C_1}} p'= \frac12(p- (1-C_1))(p- (1+C_1)). \]
 Integrating the above ODE, we arrive at the formula
 \begin{equation}\label{A-3}
 \int (p-(1+C_1))^{\frac{1}{C_1}}(p-(1-C_1))^{-\frac{1}{C_1}}dp = \frac{1}{2C_2}(t+C_3),
\end{equation}
where $C_3$ is an  arbitrary  constant.
The above quadrature can be expressed in the terms of known functions only for particular values of $C_1$.

Thus, the general solution of  the nonlinear ODE (\ref{A-1}) is
derived in the implicit form (\ref{A-2})--(\ref{A-3}).

\end{document}